\documentclass{amsart}
\usepackage{graphicx}
\newtheorem{theorem}{Theorem}[section]
\newtheorem{lemma}[theorem]{Lemma}
\newtheorem{problem}[theorem]{Problem}

\theoremstyle{definition}
\newtheorem{definition}[theorem]{Definition}

\begin{document}

\title[A locally minimal, but not globally minimal bridge position]{A locally minimal, but not globally minimal bridge position of a knot}

\author{Makoto Ozawa}
\address{Department of Natural Sciences, Faculty of Arts and Sciences, Komazawa University, 1-23-1 Komazawa, Setagaya-ku, Tokyo, 154-8525, Japan}
\email{w3c@komazawa-u.ac.jp}
\thanks{The first author is partially supported by Grant-in-Aid for Scientific Research (C) (No. 23540105), The Ministry of Education, Culture, Sports, Science and Technology, Japan}

\author{Kazuto Takao}
\address{Department of Mathematics, Graduate School of Science, Osaka University, 1-1 Machikaneyama-cho, Toyonaka, Osaka, 560-0043, Japan}
\email{kazutotakao@gmail.com}



\begin{abstract}
We give a locally minimal, but not globally minimal bridge position of a knot, that is, an unstabilized, nonminimal bridge position of a knot.
It implies that a bridge position cannot always be simplified so that the bridge number monotonically decreases to the minimal.
\end{abstract}

\maketitle

\section{Introduction}\label{intro}

A {\em knot} is an equivalence class of embeddings of the circle $S^1$ into the 3-sphere $S^3$, where two embeddings are said to be {\em equivalent} if an ambient isotopy of $S^3$ deforms one to the other.
In knot theory, it is a fundamental and important problem to determine whether given two representatives of knots are equivalent, and furthermore to describe how one can be deformed to the other.
In particular, a simplification to a ``minimal position" is of great interest.

Let $h:S^3\to \mathbb{R}$ be the standard height function, that is, the restriction of $\mathbb{R}^4=\mathbb{R}^3\times \mathbb{R}\to \mathbb{R}$ to $S^3$.
A {\em Morse position} of a knot $K$ is a representative $k$ such that $k$ is disjoint from the poles of $S^3$ and the critical points of $h|_k$ are all non-degenerate and have pairwise distinct values.
Since $k$ is a circle, there are the same number of local maxima and local minima.

In \cite{Sc}, Schubert introduced the notion of bridge position and bridge number for knots.
A {\em bridge position} of $K$ is a Morse position $k$ where all the local maxima are above all the local minima with respect to $h$.
A level $2$-sphere $S$ separating the local maxima from the local minima is called a {\em bridge sphere} of $k$.
If $k$ intersects $S$ in $2n$ points, then $k$ is called an {\em $n$-bridge position} and $n$ is called the {\em bridge number} of $k$.
The minimum of the bridge number over all bridge positions of $K$ is called the {\em bridge number} of $K$.
A knot with the bridge number $n$ is called an {\em $n$-bridge knot}.
The bridge number is a fundamental geometric invariant of knots as well as the crossing number.

In \cite{G}, Gabai introduced the notion of width for knots.
Suppose that $k$ is a Morse position of a knot $K$, let $t_1,\ldots,t_m$ be the critical levels of $h|_k$ such that $t_i<t_{i+1}$ for $i=1,\ldots,m-1$, and choose regular levels $r_1,\ldots,r_{m-1}$ of $h|_k$ so that $t_i<r_i<t_{i+1}$.
The {\em width} of $k$ is defined as $\sum_{i=1}^{m-1}|h^{-1}(r_i)\cap k|$, and the {\em width} of $K$ is the minimum of the width over all Morse positions of $K$.

Two Morse positions of a knot are {\em isotopic} if an ambient isotopy of $S^3$ deforms one to the other keeping it a Morse position except for exchanging two levels of local maxima or two levels of local minima.
Such an isotopy preserves the width of a Morse position and the bridge number of a bridge position.
The following two types of moves change the isotopy class of a Morse position.

Suppose that $k$ is a Morse position of a knot $K$ and let $t_1,r_1,t_2,\ldots,r_{m-1},t_m$ as above.
We say that the level 2-sphere $h^{-1}(r_i)$ is {\em thick} if $t_i$ is a locally minimal level of $h|_k$ and $t_{i+1}$ is a locally maximal level of $h|_k$, and that $h^{-1}(r_i)$ is {\em thin} if $t_i$ is a locally maximal level and $t_{i+1}$ is a locally minimal level.
A {\em strict upper} (resp. {\em lower}) {\em disk} for a thick sphere $S$ is a disk $D\subset S^3$ such that the interior of $D$ does not intersect with $k$ and any thin sphere, the interior of $D$ contains no critical points with respect to $h$, and $\partial D$ consists of a subarc $\alpha$ of $k$ and an arc in $S-k$.
Note that the arc $\alpha$ has exactly one local maximum (resp. minimum).
First suppose that there exist a strict upper disk $D_+$ and a strict lower disk $D_-$ for a thick sphere $S$ such that $D_+\cap D_-$ consists of a single point of $k\cap S$.
Then $k$ can be isotoped along $D_+$ and $D_-$ to cancel the local maximum in $\partial D_+$ and the local minimum in $\partial D_-$.
In \cite{Sch2}, Schultens called this move a {\em Type I move}.
The inverse operation of a Type I move is called a {\em stabilization} (\cite{B}, \cite{O1}) or a {\em perturbation} (\cite{ST}, \cite{T}) and the resultant position is said to be {\em stabilized} or {\em perturbed}.
Next suppose that there exist a strict upper disk $D_+$ and a strict lower disk $D_-$ for a thick sphere $S$ such that $D_+\cap D_-=\emptyset$. 
Then $k$ can be isotoped along $D_+$ and $D_-$ to exchange the two levels of the local maximum in $\partial D_+$ and the local minimum in $\partial D_-$.
Schultens (\cite{Sch2}) called this move a {\em Type II move}.

\begin{figure}[htbp]
	\begin{center}
	\includegraphics[bb=0 0 505 145, width=.8\linewidth]{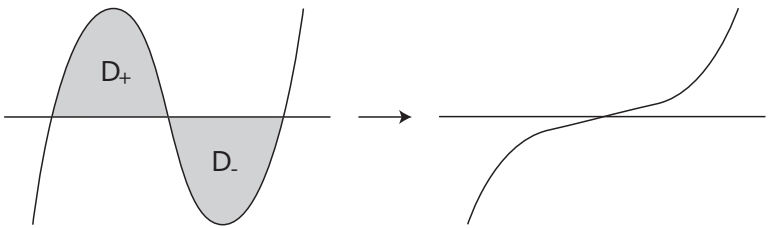}
	\end{center}
	\caption{Type I move}
	\label{type1}
\end{figure}

\begin{figure}[htbp]
	\begin{center}
	\includegraphics[bb=0 0 505 145, width=.8\linewidth]{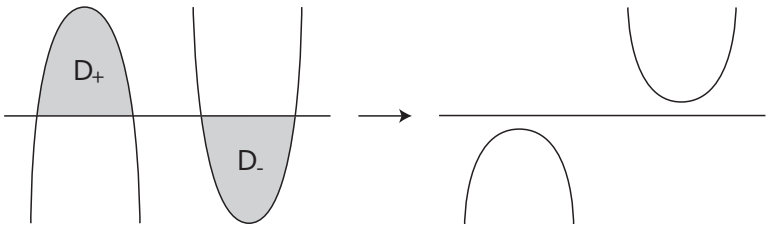}
	\end{center}
	\caption{Type II move}
	\label{type2}
\end{figure}

The following are fundamental theorems for bridge positions and Morse positions which correspond to Reidemeister's theorem (\cite{R2}, \cite{AB}) for knot diagrams.

\begin{theorem}[\cite{B}, \cite{Ha}]\label{fundamental1}
Two knots are equivalent if and only if their two bridge positions can be related by a sequence of Type I moves and the inverse operations up to isotopy.
\end{theorem}

\begin{theorem}[\cite{Sch2}]\label{fundamental2}
Two knots are equivalent if and only if their two Morse positions can be related by a sequence of Type I and Type II moves and the inverse operations up to isotopy.
\end{theorem}

We say that a bridge position of a knot $K$ is {\em globally minimal} if it realizes the bridge number of $K$, 
and a bridge position is {\em locally minimal} if it does not admit a Type I move.
Similarly, we say that a Morse position of a knot $K$ is {\em globally minimal} if it realizes the width of $K$, 
and a Morse position is {\em locally minimal} if it does not admit a Type I move nor a Type II move.
Note that if a bridge (resp. Morse) position of a knot is globally minimal, then it is locally minimal.
Otal (later Hayashi--Shimokawa, the first author) proved the converse for bridge positions of the trivial knot.

\begin{theorem}[\cite{O1}, \cite{HS}, \cite{OZ}]\label{Otal}
A locally minimal bridge position of the trivial knot is globally minimal.
\end{theorem}

This implies that even complicated bridge positions of the trivial knot can be simplified into the $1$-bridge position only by Type I moves.
Furthermore, Otal (later Scharlemann--Tomova) showed that the same statement for 2-bridge knots is true (\cite{O2}, \cite{ST}), and the first author showed that the same statement for torus knots is also true (\cite{OZ2}). 
Then, the following problem is naturally proposed.

\begin{problem}[\cite{OZ2}]\label{question}
Is any locally minimal bridge position of any knot globally minimal?
\end{problem}

In this paper, we give a negative answer to this problem.
It implies that a bridge position cannot always be simplified into a minimal bridge position only by Type I moves.

\begin{theorem}\label{main}
A $4$-bridge position $\kappa $ of a knot $\mathcal{K}$ in Figure \ref{example} is locally minimal, but not globally minimal.
\end{theorem}

\begin{figure}[ht]
\begin{minipage}{180pt}
\begin{center}
$\kappa $\\
\includegraphics[bb=0 0 180 130]{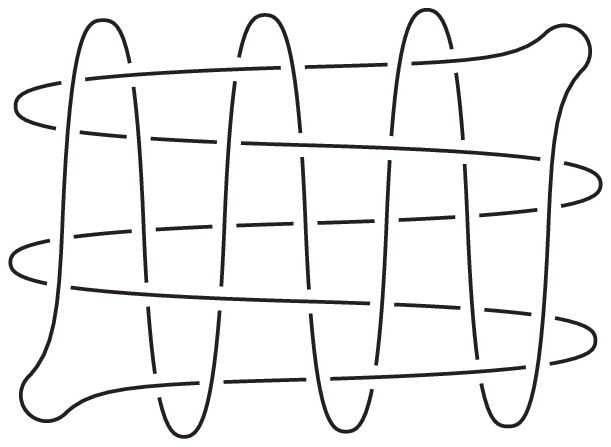}
\end{center}
\end{minipage}
\hspace{10pt}
{\LARGE $\xrightarrow{h}$}
\hspace{5pt}
\begin{minipage}{10pt}
\begin{center}
$\mathbb{R}$\\
\rotatebox{90}{$\xrightarrow{\hspace{120pt}}$}
\end{center}
\end{minipage}
\caption{A $4$-bridge position of a knot.}
\label{example}
\end{figure}

To prove this theorem, we show that the Hempel distance of the 4-bridge position is greater than $1$ by the method developed by the second author (\cite{Ta}).
This guarantees that the 4-bridge position is locally minimal (see Lemma \ref{minimal}).

On the other hand, Zupan showed that locally minimal, but not globally minimal Morse positions exist even if the knot is trivial.

\begin{theorem}[\cite{Z}]
There exists a locally minimal Morse position of the trivial knot which is not globally minimal.
\end{theorem}

We remark that this answers Scharlemann's question \cite[Question 3.5]{S1}.
By using this example, Zupan showed that there exist infinitely many locally minimal, but not globally minimal Morse positions for any knot.

\section{Proof of the main theorem}

Note that Figure \ref{example} displays a $3$-bridge position of $\mathcal{K}$ after a $(\pi /2)$-rotation of $\kappa $, and so the $4$-bridge position is not globally minimal.
To prove that the $4$-bridge position is locally minimal, we apply the following:

\begin{lemma}\label{minimal}
An $n$-bridge position is locally minimal if it has Hempel distance greater than $1$.
\end{lemma}

\begin{theorem}[\cite{Ta}]\label{criterion}
An $n$-bridge position has Hempel distance greater than $1$ if a bridge diagram of it satisfies the well-mixed condition.
\end{theorem}

\noindent
Here $n$ is an integer greater than $2$.
The notions of {\em Hempel distance}, {\em bridge diagram} and {\em well-mixed condition} are described in the following subsections.

\subsection{Hempel distance}

Suppose that $k$ is an $n$-bridge position of a knot $K$ and $S$ is a bridge sphere of $k$.
Let $B_+,\ B_-\subset S^3$ be the $3$-balls divided by $S$, and $\tau _\varepsilon $ be the $n$ arcs $k\cap B_\varepsilon $ for each $\varepsilon =\pm $.

Consider a properly embedded disk $E$ in $B_\varepsilon $.
We call $E$ an {\em essential disk} of $(B_\varepsilon ,\tau _\varepsilon )$ if $E$ is disjoint from $\tau _\varepsilon $ and $\partial E$ is essential in the $2n$-punctured sphere $S\setminus k$.
Here, a simple closed curve on a surface is said to be {\em essential} if it neither bounds a disk nor is peripheral in the surface.
The essential simple closed curves on $S\setminus k$ form a $1$-complex ${\mathcal C}(S\setminus k)$, called the {\em curve graph} of $S\setminus k$.
The vertices of ${\mathcal C}(S\setminus k)$ are the isotopy classes of essential simple closed curves on $S\setminus k$, and a pair of vertices spans an edge of ${\mathcal C}(S\setminus k)$ if the corresponding isotopy classes can be realized as disjoint curves.
The {\em Hempel distance} of $k$ is defined as
$${\rm min}\{ d([\partial E_+],[\partial E_-])\mid E_\varepsilon \text{ is an essential disk of }(B_\varepsilon ,\tau _\varepsilon )\text{ for each }\varepsilon =\pm .\} ,$$
where $d([\partial E_+],[\partial E_-])$ is the minimal distance between $[\partial E_+]$ and $[\partial E_-]$ measured in ${\mathcal C}(S\setminus k)$ with the path metric.

Assume that $k$ has Hempel distance $0$.
By the definition, there exist essential disks $E_+,\ E_-$ of $(B_+,\tau _+)$, $(B_-,\tau _-)$, respectively, such that $\partial E_+=\partial E_-$, which requires that $k$ is split.
Since the circle $k$ is connected, the Hempel distance is at least $1$.
The Hempel distance is $1$ if there exist essential disks $E_+,\ E_-$ of $(B_+,\tau _+),\ (B_-,\tau _-)$, respectively, such that $\partial E_+\cap \partial E_-=\emptyset$.
We can find such disks for a not locally minimal bridge position as follows:

\begin{proof}[Proof of Lemma \ref{minimal}]
Assume that the $n$-bridge position $k$ is not locally minimal.
By definition, there exist a strict upper disk $D_+\subset B_+$ and a strict lower disk $D^1_-\subset B_-$ for $S$ such that $D_+\cap D^1_-$ consists of a single point of $k\cap S$.
Note that $\tau _-$ is $n$ arcs each of which has a single local minimum.
We can choose strict lower disks $D^2_-,\ \ldots ,\ D^n_-\subset B_-$ for $S$ such that $D^1_-,\ D^2_-,\ \ldots ,\ D^n_-$ are pairwise disjoint.
Let $\eta (D_+\cup D^1_-)$ denote a closed regular neighborhood of $D_+\cup D^1_-$ in $S^3$.
By replacing subdisks of $D^2_-,\ \ldots ,\ D^n_-$ with subdisks of $\partial (\eta (D_+\cup D^1_-))\cap B_-$, we can arrange that $D^1_-,\ D^2_-,\ \ldots ,\ D^n_-$ are disjoint from $D_+$ except for the two points of $k\cap S$.
Since we assumed $n>2$, one of the strict lower disks, denoted by $D_-$, is disjoint from $D_+$.
The boundary of a regular neighborhood in $S^3$ of each $D_\varepsilon $ intersects $B_\varepsilon $ in an essential disk of $(B_\varepsilon ,\tau _\varepsilon )$.
They guarantee that the Hempel distance is $1$.
\end{proof}

\subsection{Bridge diagram}

We continue with the above notation.
There are $n$ pairwise disjoint strict upper (resp. lower) disks $D^1_+,\ D^2_+,\ \ldots ,\ D^n_+$ (resp. $D^1_-,\ D^2_-,\ \ldots $, $D^n_-$) for $S$.
The knot diagram of $K$ obtained by projecting $k$ into $S$ along these disks is called a {\em bridge diagram} of $k$.
In the terminology of \cite{CF}, $\tau _+,\ \tau _-$ are the overpasses and the underpasses of $k$.

Now let us describe how we can obtain a bridge diagram of the $4$-bridge position $\kappa $.
Isotope $\kappa $ as in Figure \ref{braid}, and start with a bridge sphere $S=S_0$.
There are canonical strict upper disks $D^1_+,\ D^2_+,\ D^3_+$ and $D^4_+$.
Figure \ref{diagram-1} illustrates a view of the arcs $D^1_+\cap S,\ D^2_+\cap S,\ D^3_+\cap S$ and $D^4_+\cap S$ on $S$ from $B_+$ side.
Shifting the bridge sphere $S$ to $S_1$, the arcs can be seen as in Figure \ref{diagram-2}.
Shifting $S$ further to $S_2$ and to $S_3$, the arcs are as in Figure \ref{diagram-3} and \ref{diagram-4}, respectively.
By continuing this process, the arcs are as in Figure \ref{diagram-5} when $S$ is at $S_{15}$.
The picture grows more and more complicated as $S$ goes down.
We include huge pictures in the back of this paper.
Figure \ref{diagram-6} illustrates the arcs when $S$ is at $S_{20}$, and finally Figure \ref{diagram-7} when $S$ has arrived at $S_{25}$.
Then we can find canonical strict lower disks $D^1_-,\ D^2_-,\ D^3_-,\ D^4_-$ and obtain a bridge diagram of $\kappa $.

\begin{figure}[ht]
\begin{minipage}{190pt}
\includegraphics[bb=0 0 170 305]{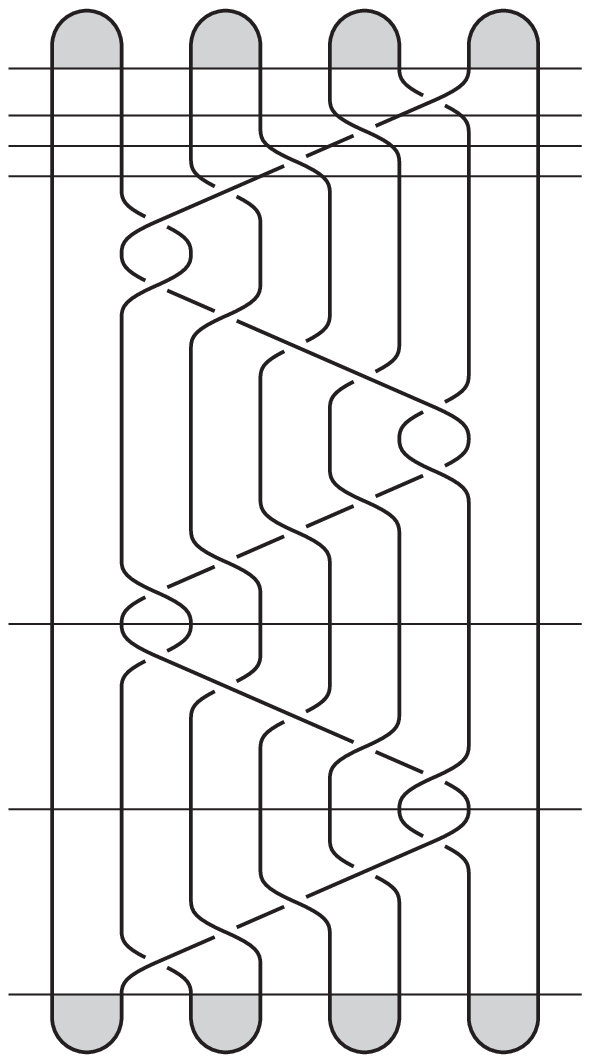}\\[-303pt]
\hspace*{18pt}$D^1_+$\hspace{25pt}$D^2_+$\hspace{25pt}$D^3_+$\hspace{25pt}$D^4_+$\\[-6pt]
\hspace*{170pt}$S_0$\\[2pt]
\hspace*{170pt}$S_1$\\[-2pt]
\hspace*{170pt}$S_2$\\[-2pt]
\hspace*{170pt}$S_3$\\[115pt]
\hspace*{170pt}$S_{15}$\\[41pt]
\hspace*{170pt}$S_{20}$\\[42pt]
\hspace*{170pt}$S_{25}$\\[-4pt]
\hspace*{18pt}$D^1_-$\hspace{25pt}$D^2_-$\hspace{25pt}$D^3_-$\hspace{25pt}$D^4_-$
\vspace{2pt}
\caption{A $4$-bridge position isotopic to $\kappa $.}
\label{braid}
\end{minipage}
\begin{minipage}{160pt}
\begin{center}
$D^1_+\cap S$\hspace{8pt}$D^2_+\cap S$\hspace{8pt}$D^3_+\cap S$\hspace{8pt}$D^4_+\cap S$\\[-15pt]
\includegraphics[bb=0 0 145 30]{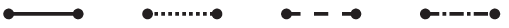}
\caption{The arcs in $S_0$.}
\label{diagram-1}
\vspace{30pt}
\includegraphics[bb=0 0 145 30]{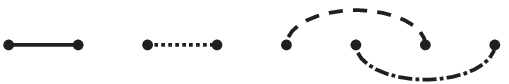}
\caption{In $S_1$.}
\label{diagram-2}
\vspace{30pt}
\includegraphics[bb=0 0 145 30]{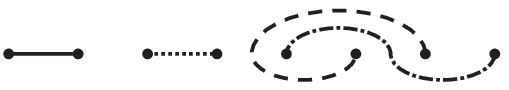}
\caption{In $S_2$.}
\label{diagram-3}
\vspace{30pt}
\includegraphics[bb=0 0 145 30]{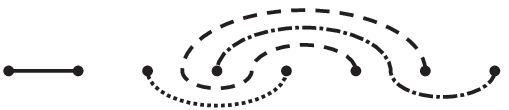}
\caption{In $S_3$.}
\label{diagram-4}
\end{center}
\end{minipage}
\end{figure}

\begin{figure}[ht]
\includegraphics[bb=0 0 260 150]{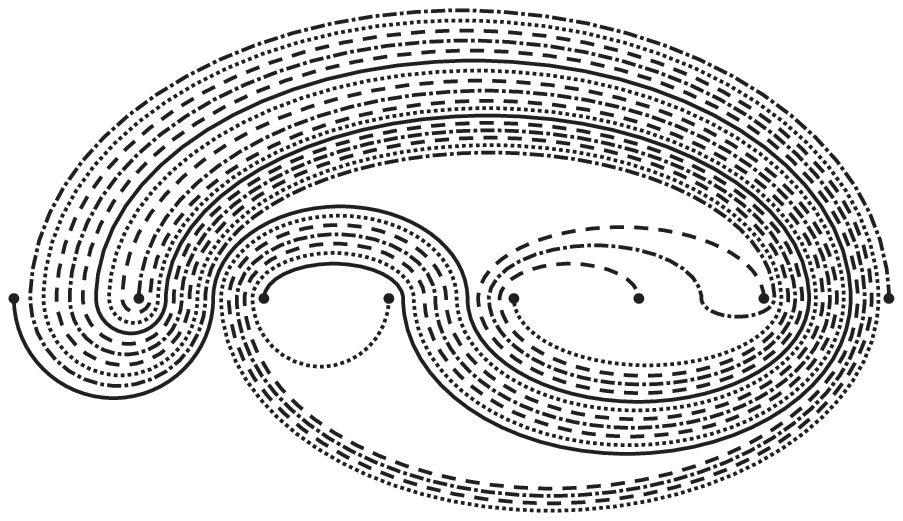}
\caption{In $S_{15}$.}
\label{diagram-5}
\end{figure}

\addtocounter{figure}{2}

\subsection{Well-mixed condition}

Suppose again that $k$ is an $n$-bridge position of a knot $K$ with $n>2$ and $S$ is a bridge sphere of $k$.
Let $B_+,\ B_-\subset S^3$ be the $3$-balls divided by $S$, and $\tau _\varepsilon $ be the $n$ arcs $K\cap B_\varepsilon $ for each $\varepsilon =\pm $.
Let $D^1_+,\ D^2_+,\ \ldots ,\ D^n_+$ and $D^1_-,\ D^2_-,\ \ldots ,\ D^n_-$ be strict upper and lower disks for $S$ determining a bridge diagram of $k$.

Let $l$ be a loop on $S$ containing the arcs $D^1_-\cap S,\ D^2_-\cap S,\ \ldots ,\ D^n_-\cap S$ such that the arcs are located in $l$ in that order.
We can assume that $D^1_+,\ D^2_+,\ \ldots ,\ D^n_+$ have been isotoped so that the arcs $D^1_+\cap S,\ D^2_+\cap S,\ \ldots ,\ D^n_+\cap S$ have minimal intersection with $l$.
For the bridge diagram of Figure \ref{diagram-7}, it is natural to choose $l$ to be the one which is seen as a horizontal line.
Let $H_+,\ H_-\subset S$ be the hemi-spheres divided by $l$ and let $\delta _i$ ($1\leq i\leq n$) be the component of $l\setminus (D^1_-\cup D^2_-\cup \cdots \cup D^n_-)$ which lies between $D^i_-\cap S$ and $D^{i+1}_-\cap S$.
(Here the indices are considered modulo $n$.)
Let ${\mathcal A}_{i,j,\varepsilon }$ be the collection of components of $(D^1_+\cup D^2_+\cup \cdots \cup D^n_+)\cap H_\varepsilon $ separating $\delta _i$ from $\delta _j$ in $H_\varepsilon $ for a distinct pair $i,\ j\in \{ 1,2,\ldots ,n\} $ and $\varepsilon \in \{ +,-\} $.
For example, Figure \ref{(1,2,+)} roughly displays ${\mathcal A}_{1,2,+}$ for the bridge diagram of Figure \ref{diagram-7}.
Note that ${\mathcal A}_{i,j,\varepsilon }$ consists of parallel arcs in $H_\varepsilon $.

\begin{figure}[ht]
\includegraphics[bb=0 0 330 140]{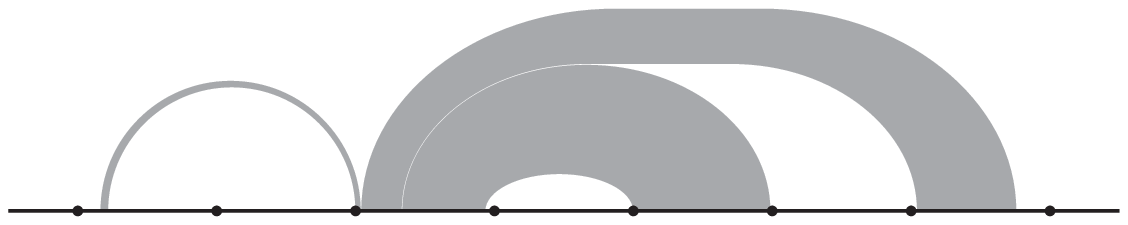}\\[-125pt]
\hspace*{-250pt}$H_+$\\[35pt]
\hspace*{10pt}$\delta _4$\hspace{10pt}$D^1_-\cap S$\hspace{21pt}$\delta _1$\hspace{18pt}$D^2_-\cap S$\hspace{20pt}$\delta _2$\hspace{18pt}$D^3_-\cap S$\hspace{21pt}$\delta _3$\hspace{18pt}$D^4_-\cap S$\hspace{11pt}$\delta _4$\hspace{8pt}\raisebox{9pt}{$l$}\\[17pt]
\hspace*{-250pt}$H_-$\\[5pt]
\caption{The collection ${\mathcal A}_{1,2,+}$ of arcs, which looks like gray bands.}
\label{(1,2,+)}
\end{figure}

\begin{definition}
\begin{enumerate}
\item A bridge diagram satisfies the {\em $(i,j,\varepsilon )$-well-mixed condition} if in ${\mathcal A}_{i,j,\varepsilon }\subset H_\varepsilon $, a subarc of $D^r_+\cap S$ is adjacent to a subarc of $D^s_+\cap S$ for every distinct pair $r,\ s\in \{ 1,2,\ldots ,n\} $.
\item A bridge diagram satisfies the {\em well-mixed condition} if it satisfies the $(i,j,\varepsilon )$-well-mixed condition for every combination of a distinct pair $i,\ j\in \{ 1,2,\ldots ,n\} $ and $\varepsilon \in \{ +,-\} $.
\end{enumerate}
\end{definition}

By making Figure \ref{(1,2,+)} detailed, one can check the $(1,2,+)$-well-mixed condition for the bridge diagram of Figure \ref{diagram-7}.
One can also check the $(i,j,\varepsilon )$-well-mixed condition for all the other $(i,j,\varepsilon )$ to complete the well-mixed condition.
By Theorem \ref{criterion}, the Hempel distance of $\kappa $ is greater than $1$.
By Lemma \ref{minimal}, we conclude the proof of Theorem \ref{main}.

We would like to remark that the Hempel distance of $\kappa $ is exactly $2$.
Notice that the boundary of a regular neighborhood in $S$ of the closure of $\delta _3$ is a simple closed curve disjoint from both $D^1_+$ and $D^1_-$.
Note that the boundary of a regular neighborhood in $S^3$ of each $D^1_\varepsilon $ intersects $B_\varepsilon $ in an essential disk of $(B_\varepsilon ,\tau _\varepsilon )$.
They guarantee that the Hempel distance is at most $2$.

\section{Related results and further directions}

\subsection{The knot of our example}

Figure \ref{example} shows that the bridge number of $\mathcal{K}$ is at most $3$.
Since any locally minimal bridge position of any 2-bridge knot is globally minimal (\cite{O2}, \cite{ST}), the bridge number of $\mathcal{K}$ is equal to $3$.
Since $\mathcal{K}$ has a $4$-bridge position with the Hempel distance $2$, it is a prime knot by the following:

\begin{theorem}\label{composite}
A bridge position of a composite knot has Hempel distance $1$.
\end{theorem}

\begin{proof}
Let $k$ be an $n$-bridge position of a composite knot and $S$ be a bridge sphere of $k$.
Let $B_+,\ B_-\subset S^3$ be the $3$-balls divided by $S$, and $\tau _\varepsilon $ be the $n$ arcs $k\cap B_\varepsilon $ for each $\varepsilon =\pm $.
By the arguments in \cite{Sc}, \cite{Sch1}, it follows that any decomposing sphere for $k$ can be isotoped so that it intersects $S$ in a single loop.
Then, in the opposite sides of the decomposing sphere, there are two essential disks $E_+,\ E_-$ of $(B_+,\tau _+),\ (B_-,\tau _-)$ respectively such that $\partial E_+\cap \partial E_-=\emptyset$.
This shows that the Hempel distance is $1$.
\end{proof}

\noindent
Furthermore, $\mathcal{K}$ is hyperbolic since any locally minimal bridge position of any torus knot is globally minimal (\cite{OZ2}).
Thus, $\mathcal{K}$ is a hyperbolic 3-bridge knot which admits a $4$-bridge position with the Hempel distance 2.

We expect that not only $\kappa $ may be an example for Theorem \ref{main} but also many $4$-bridge positions of knots with the same projection image as that of Figure \ref{example}.
However only finitely many knots have the same projection image, and we would like to ask the following problem.

\begin{problem}
For an integer $n>3$, can we generate infinitely many $n$-bridge positions which are locally minimal, but not globally minimal?
\end{problem}

We further expect that for some integers $n>m\geq 3$, we can find a locally minimal $n$-bridge position of an $m$-bridge knot which has a similar projection image as that of Figure \ref{example}.
However it seems difficult to find more than two locally minimal bridge positions of such a knot, and we would like to ask the following problem.

\begin{problem}
Does any knot have infinitely many locally minimal bridge positions?
\end{problem}

It should be remarked that there exist only finitely many bridge positions of given bridge numbers for a hyperbolic knot (\cite{C}).
In particular, there are finitely many globally minimal bridge positions of a hyperbolic knot.
It should be also remarked that multiple bridge surfaces restrict Hempel distances (\cite{T2}).

\subsection{Essential surfaces}

Composite knots are a simple example of knots with essential surfaces properly embedded in the exteriors of their representatives.
Theorem \ref{composite} suggests that essential surfaces restrict Hempel distances.
Bachman--Schleimer showed it in general.

\begin{theorem}[\cite{BS}]\label{bound}
Let $F$ be an orientable essential surface properly embedded in the exterior of a bridge position $k$ of a knot.
Then the Hempel distance of $k$ is bounded above by twice the genus of $F$ plus $|\partial F|$.
\end{theorem}

By Theorem \ref{bound}, if a knot exterior contains an essential annulus or an essential torus, then the Hempel distance of a bridge position is at most 2.
Therefore, if there exists a bridge position of a knot with the Hempel distance at least 3, then the knot is hyperbolic.
The properties of our knot $\mathcal{K}$ can be compared with it.

A knot without an essential surface with meridional boundary in the exterior of its representative is called a {\em meridionally small knot}.
For example, the trivial knot, 2-bridge knots and torus knots are known to be meridionally small.
As we mentioned in Section \ref{intro}, these knots also have the nice property that any nonminimal bridge position is stabilized.
We say that a knot is {\em destabilizable} if it has this property.
Zupan showed that any cabled knot $J$ of a meridionally small knot $K$ is also meridionally small, and that if $K$ is destabilizable, then $J$ is also destabilizable (\cite{Z2}).
Then, the following problem is naturally proposed.

\begin{problem}\label{small}
Is there a relation between meridionally small knots and destabilizable knots?
\end{problem}

We remark that a bridge position of a meridionally small knot is locally minimal if and only if the Hempel distance is greater than $1$ by Lemma \ref{minimal} and the following fundamental result:

\begin{theorem}[\cite{HS2}, \cite{T}]\label{WR}
If a bridge position of a knot has Hempel distance $1$, then either it is stabilized or the knot exterior contains an essential surface with meridional boundary.
\end{theorem}

On the other hand, it is not always true that if the knot exterior contains an essential surface with meridional boundary, then a bridge position has Hempel distance $1$.
For example, \cite[Example 5.1]{HK} shows that a $3$-bridge position of $8_{16}$ has Hempel distance greater than $1$, but the knot exterior contains an essential surface with meridional boundary.

\subsection{Distance between bridge positions}

Theorem \ref{fundamental1} allows us to define a distance between two bridge positions of a knot, which we call the {\em Birman distance}.
That is to say, the Birman distance between two bridge positions is the minimum number of Type I moves and the inverse operations relating the bridge positions up to isotopy.
For example, the Birman distance between an $n$-bridge position and an $m$-bridge position of the trivial knot is always $|n-m|$ by Theorem \ref{Otal}.
The Birman distance between $\kappa $ and the $3$-bridge position of $\mathcal{K}$ is at least $3$ since $\kappa $ is locally minimal.
In fact, we can see that it is at most $5$ by observing the $(\pi /2)$-rotation of $\kappa $.

Johnson--Tomova gave an upper bound for the Birman distance between two bridge positions with high Hempel distance which are obtained from each other by flipping, namely the rotation of $S^3$ exchanging the poles.

\begin{theorem}[\cite{JT}]
For an integer $n\geq 3$, if an $n$-bridge position $k$ of a prime knot has Hempel distance at least $4n$, then the Birman distance between $k$ and the flipped bridge position of $k$ is $2n-2$.
\end{theorem}

They also gave the following, which holds even if we consider bridge positions modulo flipping.

\begin{theorem}[\cite{JT}]
For an integer $n\geq 2$, there exists a composite knot with a $2n$-bridge position and a $(2n-1)$-bridge position such that the Birman distance is at least $2n-7$.
\end{theorem}

We remark that the $2n$-bridge position is not locally minimal, and hence it does not answer Problem \ref{question}.
It turns out that there are two $(2n-1)$-bridge positions such that the Birman distance is at least $2n-6$.
The following are major problems.

\begin{problem}
Determine or estimate the Birman distance in terms of some invariants of the bridge positions.
\end{problem}

\begin{problem}
For a given $n$, does there exist a universal upper bound for the Birman distance between locally minimal bridge positions of every $n$-bridge knot?
\end{problem}

\bigskip

\noindent{\bf Acknowledgements.}
The authors are grateful to Joel~Hass for suggesting the construction of the bridge position $\kappa $.
They would also like to thank Alexander~Zupan for valuable comments.

\bibliographystyle{amsplain}

\setcounter{figure}{9}

\begin{figure}[ht]
\includegraphics[bb=0 0 179 200,clip,scale=2]{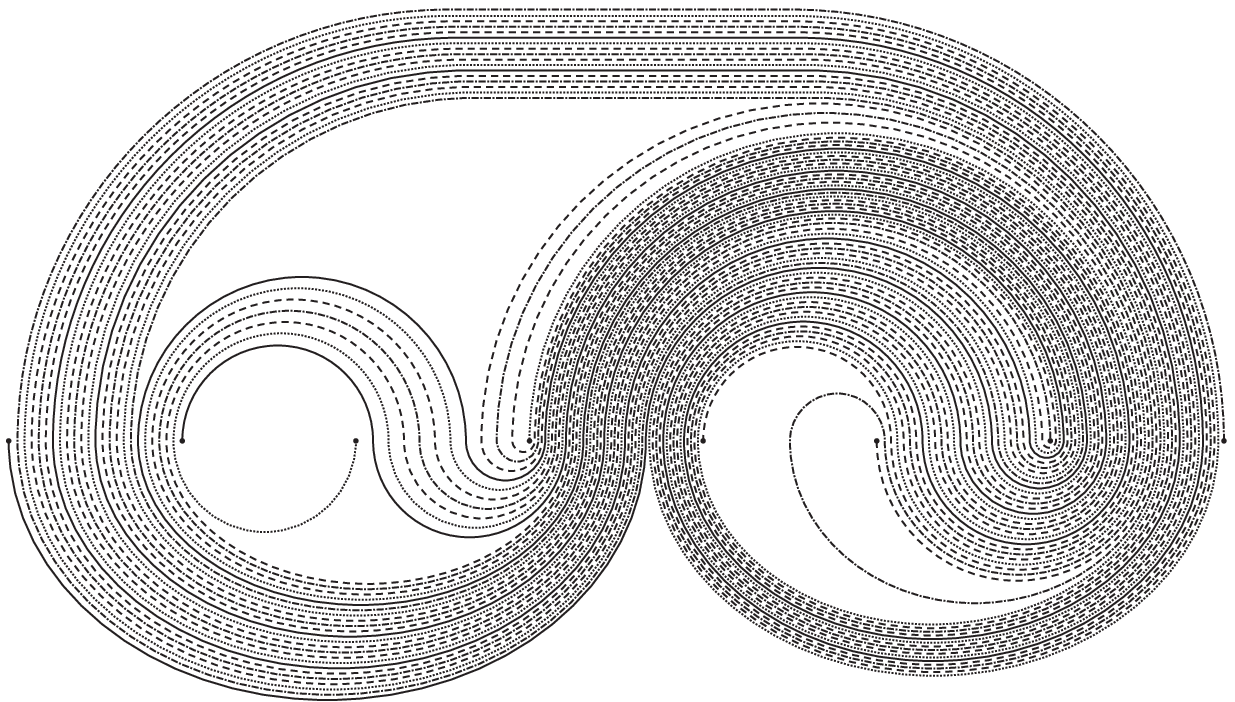}
\caption{The arcs $D^1_+\cap S,\ D^2_+\cap S,\ D^3_+\cap S$ and $D^4_+\cap S$ in $S=S_{20}$, which extend to the next page.}
\label{diagram-6}
\end{figure}
\begin{figure}[ht]
\includegraphics[bb=176 0 355 200,clip,scale=2]{diagram-6.eps}\\[10pt]
The right part of Figure \ref{diagram-6}.
\end{figure}

\begin{figure}[ht]
\includegraphics[bb=0 0 310 150]{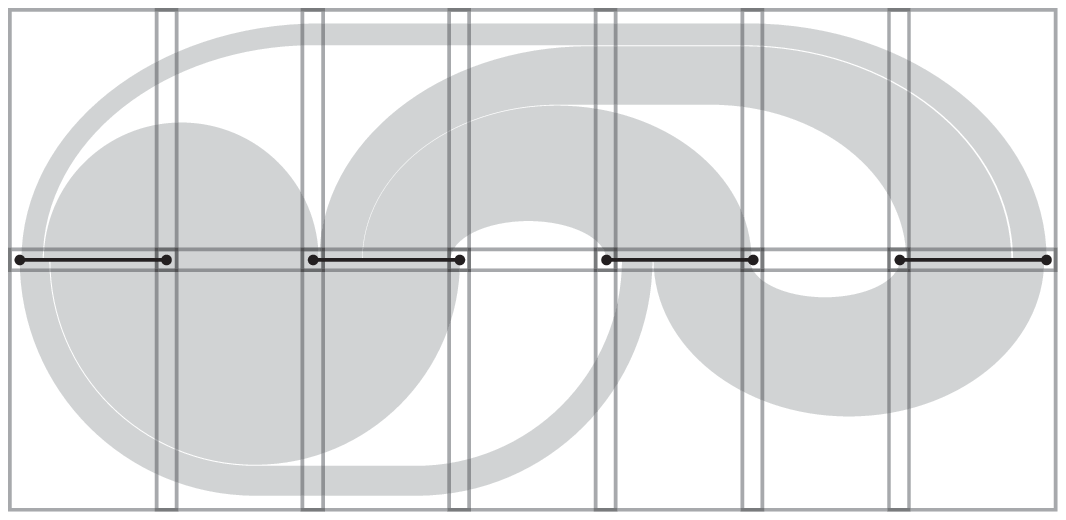}\\[-120pt]
$(+,1)$\hspace{17pt}$(+,2)$\hspace{18pt}$(+,3)$\hspace{17pt}$(+,4)$\hspace{17pt}$(+,5)$\hspace{18pt}$(+,6)$\hspace{17pt}$(+,7)$\\[60pt]
$(-,1)$\hspace{17pt}$(-,2)$\hspace{18pt}$(-,3)$\hspace{17pt}$(-,4)$\hspace{17pt}$(-,5)$\hspace{18pt}$(-,6)$\hspace{17pt}$(-,7)$\\[35pt]
\caption{A bridge diagram of the $4$-bridge position, which is decomposed into the following 14 pages.}
\label{diagram-7}
\end{figure}

\begin{figure}[p]
\includegraphics[bb=0 114 75 230,clip,scale=4.5]{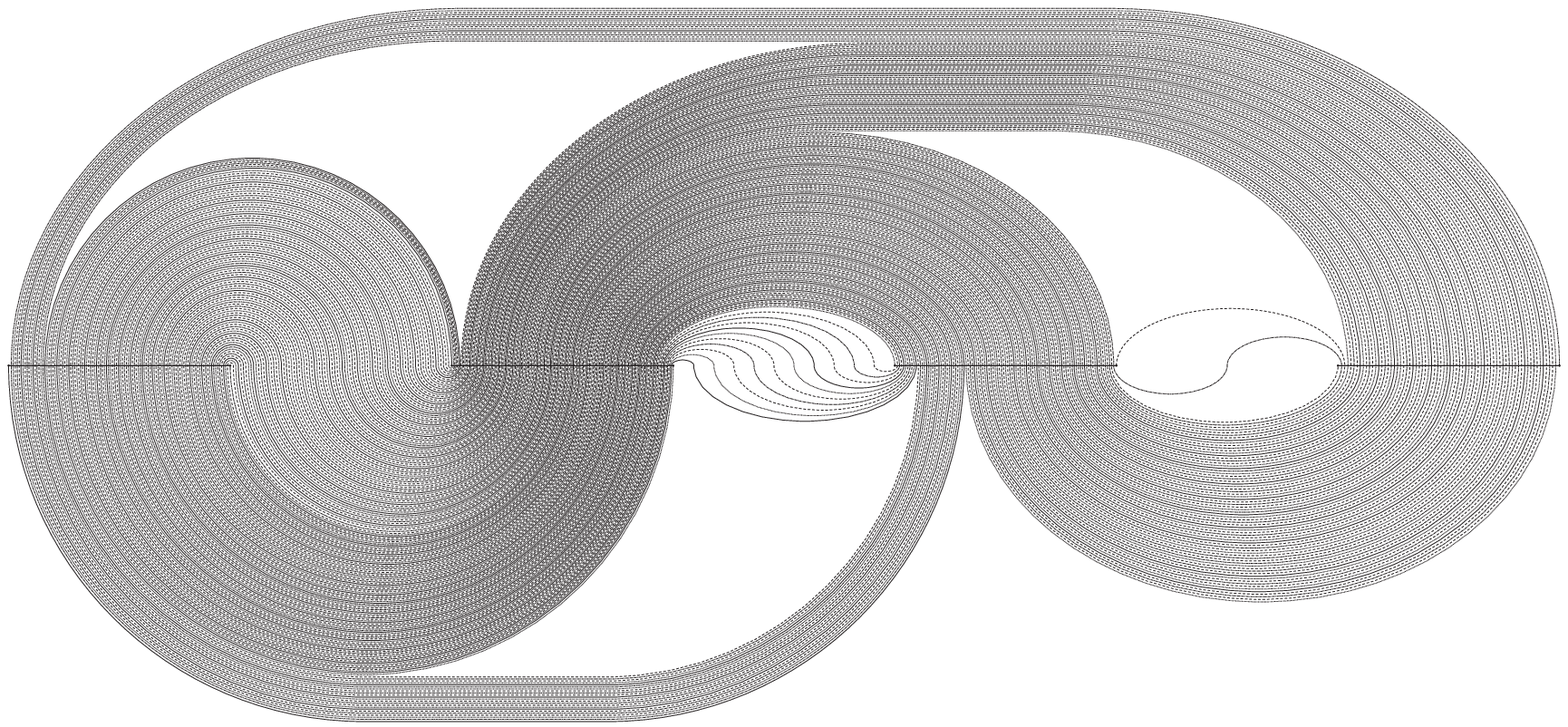}\\[10pt]
The part $(+,1)$ of Figure \ref{diagram-7}.
\end{figure}
\begin{figure}[p]
\includegraphics[bb=73 114 146 230,clip,scale=4.5]{diagram-7.eps}\\[10pt]
The part $(+,2)$ of Figure \ref{diagram-7}.
\end{figure}
\begin{figure}[p]
\includegraphics[bb=143 114 216 230,clip,scale=4.5]{diagram-7.eps}\\[10pt]
The part $(+,3)$ of Figure \ref{diagram-7}.
\end{figure}
\begin{figure}[p]
\includegraphics[bb=214 114 286 230,clip,scale=4.5]{diagram-7.eps}\\[10pt]
The part $(+,4)$ of Figure \ref{diagram-7}.
\end{figure}
\begin{figure}[p]
\includegraphics[bb=284 114 357 230,clip,scale=4.5]{diagram-7.eps}\\[10pt]
The part $(+,5)$ of Figure \ref{diagram-7}.
\end{figure}
\begin{figure}[p]
\includegraphics[bb=354 114 427 230,clip,scale=4.5]{diagram-7.eps}\\[10pt]
The part $(+,6)$ of Figure \ref{diagram-7}.
\end{figure}
\begin{figure}[p]
\includegraphics[bb=425 114 500 230,clip,scale=4.5]{diagram-7.eps}\\[10pt]
The part $(+,7)$ of Figure \ref{diagram-7}.
\end{figure}
\begin{figure}[p]
\includegraphics[bb=0 0 75 115,clip,scale=4.5]{diagram-7.eps}\\[10pt]
The part $(-,1)$ of Figure \ref{diagram-7}.
\end{figure}
\begin{figure}[p]
\includegraphics[bb=73 0 146 115,clip,scale=4.5]{diagram-7.eps}\\[10pt]
The part $(-,2)$ of Figure \ref{diagram-7}.
\end{figure}
\begin{figure}[p]
\includegraphics[bb=143 0 216 115,clip,scale=4.5]{diagram-7.eps}\\[10pt]
The part $(-,3)$ of Figure \ref{diagram-7}.
\end{figure}
\begin{figure}[p]
\includegraphics[bb=214 0 286 115,clip,scale=4.5]{diagram-7.eps}\\[10pt]
The part $(-,4)$ of Figure \ref{diagram-7}.
\end{figure}
\begin{figure}[p]
\includegraphics[bb=284 0 357 115,clip,scale=4.5]{diagram-7.eps}\\[10pt]
The part $(-,5)$ of Figure \ref{diagram-7}.
\end{figure}
\begin{figure}[p]
\includegraphics[bb=354 0 427 115,clip,scale=4.5]{diagram-7.eps}\\[10pt]
The part $(-,6)$ of Figure \ref{diagram-7}.
\end{figure}
\begin{figure}[p]
\includegraphics[bb=425 0 500 115,clip,scale=4.5]{diagram-7.eps}\\[10pt]
The part $(-,7)$ of Figure \ref{diagram-7}.
\end{figure}

\end{document}